\documentclass[preprint, 11pt, english]{elsarticle}
\usepackage{amsmath}
\usepackage{latexsym, amssymb}
\usepackage{amsthm}
\usepackage{color}
\usepackage{txfonts}

\newtheorem{thm}{Theorem}[section] 

\newtheorem{cor}[thm]{Corollary}

\newtheorem{lem}[thm]{Lemma}
\newtheorem{prop}[thm]{Proposition}

\newcommand\operA[2]{{\if!#2!\operatorname{#1}\else{\operatorname{#1}_{#2}^{\phantom{I}}}\fi}} 

\newcommand\Cref[1]{{Corollary~\ref{#1}}}%
\newcommand{\Trace}[1][]{\if!#1!\operatorname{Tr}\else{\operatorname{Tr}_{#1}^{\phantom{I}}}\fi} 

\long\def\forget#1\forgotten{{}} %

\def\({\left(}
\def\){\right)}

\newif\iffurther
\furtherfalse

\newif\ifXY 
\XYtrue     
\ifXY

\input xy
\input xyidioms.tex
\usepackage{xy}
\xyoption{all} %
\fi 

\usepackage{babel}

\journal{Archiv der Mathematik}

\begin{document}

\begin{frontmatter}

\title{Square-Central and Artin-Schreier Elements in Division Algebras}

\author[1]{Demba Barry}
\ead{barry.demba@gmail.com}

\author[2]{Adam Chapman}
\ead{adam1chapman@yahoo.com}

\address[1]{D\'{e}partement de Math\'{e}matique et Informatique , Universit\'{e} de Bamako, Colline de Badalabougou BP E3206, Bamako, Mali}

\address[2]{Department of Mathematics, Michigan State University, East Lansing, MI 48824}

\begin{abstract}
We study the behavior of square-central elements and Artin-Schreier elements in division algebras of exponent 2 and degree a power of 2.
We provide chain lemmas for such elements in division algebras over 2-fields $F$ of cohomological $2$-dimension $\operatorname{cd}_2(F) \leq 2$, and deduce a common slot lemma for tensor products of quaternion algebras over such fields.
We also extend to characteristic 2 a theorem proven by Merkurjev for characteristic not 2 on the decomposition of any central simple algebra of exponent 2 and degree a power of 2 over a field $F$ with $\operatorname{cd}_2(F) \leq 2$ as a tensor product of quaternion algebras.
\end{abstract}


\end{frontmatter}

\section{Introduction}

A quaternion algebra over a field $F$ is an algebra of the form
$$[\alpha,\beta)=F[x,y : x^2+x=\alpha, y^2=\beta, y x y^{-1}=x+1]$$
for some $\alpha \in F$ and $\beta \in F^\times$ if $\operatorname{char}(F)=2$, and
$$(\alpha,\beta)=F[x,y : x^2=\alpha, y^2=\beta, y x y^{-1}=-x]$$
for some $\alpha,\beta \in F^\times$ if $\operatorname{char}(F) \neq 2$.

The common slot lemma for quaternion algebras states that for every two isomorphic quaternion algebras $(\alpha,\beta)$ and $(\alpha',\beta')$ there exists $\beta'' \in F^\times$ such that $(\alpha,\beta) \cong (\alpha,\beta'') \cong (\alpha',\beta'') \cong (\alpha',\beta')$ (see \cite[Lemma 5.6.45 \& Corollary 5.6.48]{Jacobson1996}).

A noncentral element $v$ in a quaternion $F$-algebra is square-central if $v^2 \in F^\times$.
If $\operatorname{char}(F)=2$, a noncentral element $v$ is Artin-Schreier if $v^2+v \in F$. For every two Artin-Schreier elements $v$ and $w$ in a quaternion algebra there exists a square-central element $t$ such that $v t+t v=t w+w t=t$. If $\operatorname{char}(F)\neq 2$, the square-central elements are the ``pure quaternions". For every two square-central elements $v$ and $w$ in a quaternion algebra there exists another square-central element $t$ such that $v t=-t v$ and $t w=-w t$.
The common slot lemma is an immediate result of this fact, which we refer to as the chain lemma for square-central or Artin-Schreier elements.

We extend the notions of square-central and Artin-Schreier elements and their chain lemmas to any algebra of degree a power of 2 and exponent 2.
The notion of common slot lemma extends naturally to any tensor product of quaternion algebras.
In characteristic not 2, a common slot lemma and a chain lemma for square-central elements in tensor products of two quaternion algebras were provided in \cite{Siv} and \cite{ChapVish2}, respectively.
Equivalent results in characteristic 2 were provided in \cite[Section 3.3]{ChapmanPHD}.
In \cite{Chapman2015} a common slot lemma was provided for tensor products of quaternion algebras over fields $F$ with $\operatorname{cd}_2(F) \leq 2$.
However, a chain lemma for the square-central elements was not provided, and the length of the chain was not computed.

In this paper we prove several facts about square-central and Artin-Schreier elements, and in  particular we provide a chain lemma for them in division algebras of exponent 2 over 2-fields $F$ with $\operatorname{cd}_2(F) \leq 2$.
We then deduce the common slot lemma for tensor products of quaternion algebras over such fields and bound the length of the chains from above by 3, which is equal to the maximal length of such chains for biquaternion algebras over arbitrary fields. We make use of \citep[Theorem 3]{Kahn} on the decomposition of any central simple algebra of exponent 2 and degree a power of 2 over a field $F$ with $\operatorname{cd}_2(F) \leq 2$ as a tensor product of quaternion algebras, which was proven in that paper for characteristic not 2 and we prove it here also for characteristic 2.

\section{Preliminaries}\label{Secexact}

Quadratic forms play a major role in the study of central simple algebras of exponent 2.
By $I_q F$ we denote the group of Witt equivalence classes of even dimensional nonsingular quadratic forms over the field $F$.
Every such form is isometric to $[a_1,b_1] \perp \dots \perp [a_m,b_m]$ for some integer $m$ and $a_1,b_1,\dots,a_m,b_m \in F$ if $\operatorname{char}(F)=2$, and $\langle a_1,\dots,a_{2 m} \rangle$ for some $a_1,\dots,a_{2 m} \in F^\times$ if $\operatorname{char}(F)\neq 2$. The expression $[a,b]$ stands for the quadratic form $a u^2+u v+b v^2$, $\langle a_1,\dots,a_n \rangle$ stands for the diagonal form $a_1 u_1^2+\dots+a_n u_n^2$, and $\perp$ is the orthogonal sum of forms. 

The discriminant (also known as the Arf invariant in characteristic 2) is denoted by $\delta$ and defined as follows:
If $\operatorname{char}(F)=2$, $\delta$ maps $I_q F$ to the additive group $F/\{a^2+a : a \in F\}$ by $$\delta([a_1,b_1] \perp \dots \perp [a_m,b_m])=a_1 b_1+\dots+a_m b_m.$$
If $\operatorname{char}(F) \neq 2$, $\delta$ maps $I_q F$ to the multiplicative group $F^\times/(F^\times)^2$ by $$\delta(\langle a_1,\dots,a_{2 m}\rangle)=(-1)^m a_1 \dots a_{2 m}$$
(see \cite[Section 13]{EKM}).

The subgroup of $I_q F$ of forms with trivial discriminant is denoted by $I_q^2 F$.
We follow the traditional abuse of notation of writing $f \in I_q^2 F$ when we actually mean that the Witt equivalence class of $f$ belongs to $I_q^2 F$.

The Clifford algebra of a quadratic form $f$ is defined to be
\begin{eqnarray*}
C(f)=F[x_1,\dots,x_{2 m} : (u_1 x_1+\dots+u_{2 m} x_{2 m})^2=\\=f(u_1,\dots,u_{2 m})\quad \forall u_1,\dots,u_{2 m} \in F].
\end{eqnarray*}
For $f \in I_q F$ with $\dim(f)=2 m$, this algebra is a central simple algebra of exponent 2 and degree $2^m$. The Clifford invariant $c(f)$ of $f$ is the Brauer class of $C(f)$ in $Br_2(F)$. Restricted to $I_q^2 F$, the Clifford invariant is an epimorphism from $I_q^2 F$ to $Br_2(F)$ with kernel $I_q^3 F$ (see \cite{Merkurjev1981} for $\operatorname{char}(F) \neq 2$ and \cite{Sah1972} for $\operatorname{char}(F)=2$). For $f \in I_q^2 F$ one has $C(f) \cong M_2(E(f))$ with $E(f)$ a central simple algebra over $F$ of degree $2^{m-1}$ isomorphic to a tensor product of quaternion algebras.

The cohomological 2-dimension of $F$, denoted by $\operatorname{cd}_2(F)$, is by definition $\leq 2$ if $H^{n+1,n}(L,\mathbb{Z}/2 \mathbb{Z})=0$ for every finite field extension $L/F$ and $n \geq 2$ (see \cite[Section 101.B]{EKM}). The latter holds if and only if $I_q^3 L=0$ (see \cite[Fact 16.2]{EKM}). If $\operatorname{cd}_2(F) \leq 2$ then $I_q^2 F \cong Br_2(F)$.

A field extension $K/F$ is excellent if for every form $f$ over $F$ there exists a quadratic form $f'$ over $F$ such that $f'_K$ is isometric to the anisotropic part of $f_K$. In particular every quadratic field extension is excellent (see \cite[Example 29.2]{EKM} and \cite[Lemma 1]{MammoneMoresi}).
A 2-field is a field with no nontrivial odd degree extensions. (In \cite{EKM} such fields are called 2-special.)
In \cite[Theorem 4]{Merkurjev1991} interesting examples of fields $F$ with $\operatorname{cd}_2(F)=2$ were constructed. The odd closure $F'$ of such $F$ also has $\operatorname{cd}_2(F')=2$ (see \cite[Example 101.17]{EKM}).

The following lemma will be used later on:
\begin{lem}\label{exact}
For any field extension $K/F$, if $f \in I_q F$ and $f_K \in I_q^2 K$ then there exists $f' \in I_q^2 F$ such that $f'_K \simeq f_K$.
\end{lem}

\begin{proof}
Let $\delta$ be a representative of the discriminant of $f$.

Assume $\operatorname{char}(F)=2$.
On the one hand $\delta \in F$. On the other hand, $\delta=u^2+u$ for some $u \in K$ because $f_K \in I_q^2 K$.
Now, $f \simeq [a_1,b_1] \perp \dots \perp [a_n,b_n]$ for some $a_1,b_1,\dots,a_n,b_n \in F$.
Without loss of generality we can assume $a_1 \neq 0$.
Set $f'=[a_1,b_1+a_1^{-1} \delta] \perp \dots \perp [a_n,b_n]$.
Because of the discriminant, $f' \in I_q^2 F$.

We shall now prove that $f'_K \simeq f_K$.
Since all the summands are the same except the first one, we shall prove that under the scalar extension, the first summands in both forms are isometric.
\begin{eqnarray*}
[a_1,b_1+a_1^{-1} \delta]_K=a_1 x^2+x y+(b_1+a_1^{-1} \delta) y^2 \simeq a_1 (x^2+x y+(a_1 b_1+\delta) y^2) \simeq\\
a_1 ((x+u y)^2+(x+u y) y+(a_1 b_1+\delta) y^2)=a_1 (x^2+x y+a_1 b_1 y^2) \simeq [a_1,b_1]_K
\end{eqnarray*}
(The first isometry is obtained by replacing $y$ with $a_1 y$.)

Assume $\operatorname{char}(F)\neq 2$.
On the one hand $\delta \in F^\times$. On the other hand, $\delta=u^2$ for some $u \in K^\times$ because $f_K \in I_q^2 K$.
Now, $f \simeq \langle a_1, \dots ,a_n \rangle$ for some $a_1,\dots,a_n \in F^\times$.
Set $f'=\langle \delta^{-1} a_1,a_2 \dots ,a_n \rangle$.
Because of the discriminant, $f' \in I_q^2 F$.
It is obvious that $f'_K \simeq f_K$.
\end{proof}

\section{Square-central and Artin-Schreier elements}

Let $A$ be a division algebra of exponent 2 and degree $2^n$ over a field $F$ for some $n \geq 2$.
In this section we prove a couple of facts about square-central and Artin-Schreier elements in $A$ without restricting the cohomological dimension of $F$.
If $\operatorname{char}(F)\neq 2$ and $x$ is square-central then any element $t$ in the algebra decomposes into $t_0+t_1$ such that $t_0=\frac{1}{2} (t+x t x^{-1})$ commutes with $x$ and $t_1=\frac{1}{2} (t-x t x^{-1})$ anti-commutes with $x$.
If $\operatorname{char}(F)=2$ and $x$ is Artin-Schreier then any element $t$ in the algebra decomposes into $t_0+t_1$ such that $t_0=x t+t x+t$ commutes with $x$ and $t_1=x t+t x$ satisfies $x t_1+t_1 x=t_1$.

\begin{lem}\label{deg4}
Assume $A$ is a division algebra of degree 4 and exponent 2. Let $x$ and $x'$ be two commuting elements in $A$ with $F[x] \neq F[x']$.
If $\operatorname{char}(F)=2$ then assume $x$ is either Artin-Schreier or square-central and $x'$ is Artin-Schreier.
If $\operatorname{char}(F)\neq 2$ assume that $x$ and $x'$ are square-central.
Then $A$ decomposes as $Q_1 \otimes Q_2$ such that $x \in Q_1$ and $x' \in Q_2$.
\end{lem}

\begin{proof}
Assume $\operatorname{char}(F)=2$.
The involution on $F[x,x']$ mapping $x'$ to $x'+1$ and $x$ to $x+1$ if $x$ is Artin-Schreier and to $x$ if $x$ is square-central extends to $A$. Denote this involution by $\tau$. Then $\tau$ restricts to an involution of the second kind on the centralizer $C(x')$ (a quaternion algebra over $F[x']$), and it commutes with the canonical involution $\gamma$ on this algebra. Therefore, $\gamma \circ \tau$ is an automorphism of order 2 on $C(x')$, and the algebra of fixed points is a quaternion algebra $Q_1$ over $F$ containing $x$ and centralizing $x'$. (See \cite[Proposition 2.22]{BOI}.)

The proof in characteristic not 2 can be written in a similar way, or concluded from \cite[Proposition 3.9]{ChapVish2}.
\end{proof}

\begin{prop}\label{charnot2commuting}
Assume $F$ is a 2-field and $A$ is a division algebra of exponent 2 and degree $2^n$ over $F$.
If $\operatorname{char}(F)=2$ then for any two Artin-Schreier elements in $A$ there exists an element, either Artin-Schreier or square-central, commuting with them both. If $\operatorname{char}(F)\neq 2$, for any two square-central elements in $A$ there exists a square-central element commuting with them both.
\end{prop}

\begin{proof}
Assume $\operatorname{char}(F)=2$.
Let $x,t$ be two Artin-Schreier elements in $A$.
Now, $t=t_0+t_1$ where $t_0$ commutes with $x$ and $t_1$ satisfies $x t_1+t_1 x=t_1$.
The element $t^2+t=t_0^2+t_1 t_0+t_0 t_1+t_1^2+t_0+t_1$ decomposes similarly as $T_0+T_1$ where $T_0$ commutes with $x$ and $T_1$ satisfies $x T_1+T_1 x=T_1$. Clearly $T_0=t_0^2+t_1^2+t_0$ and $T_1=t_1 t_0+t_0 t_1+t_1$. 
However, $t^2+t$ is central, which means that $T_1=0$ and so $t_1 t_0+t_0 t_1=t_1$.
If $t_1=0$ then $x$ and $t$ commute and the statement is trivial.
If $t_1 \neq 0$ then $F[x,t]=K[x,t_1]$ where $K$ is the field extension generated over $F$ by $t_1^2$ and $x+t_0$. Clearly $K[x,t_1]$ is a quaternion algebra over $K$.
If the $K=F$ then $A=F[x,t] \otimes A_0$. Now $A_0$ is of degree $2^{n-1}$ and therefore it contains a field extension of degree $2^{n-1}$ of the center, and since $F$ is a 2-field it contains a quadratic field extension of $F$ (see \cite[Proposition 101.15]{EKM}), which is generated by either a square-central or Artin-Schreier element. This element commutes with $x$ and $t$.
If $K \neq F$ then it must be of degree a power of 2.
Since $F$ is a 2-field, $K$ must contain a quadratic field extension of $F$ which is generated by either a square-central or Artin-Schreier element. This element commutes with $x$ and $t$.

Assume $\operatorname{char}(F) \neq 2$.
Let $x,t$ be two square-central elements in $A$.
Now, $t=t_0+t_1$ where $t_0$ commutes with $x$ and $t_1$ anti-commutes with $x$.
The element $t^2=t_0^2+t_1 t_0+t_0 t_1+t_1^2$ decomposes similarly as $T_0+T_1$ where $T_0$ commutes with $x$ and $T_1$ satisfies $x T_1=-T_1 x$. Clearly $T_0=t_0^2+t_1^2$ and $T_1=t_1 t_0+t_0 t_1$. 
However, $t^2$ is central, which means that $T_1=0$ and so $t_1 t_0=-t_0 t_1$.
If $t_1=0$ then $x$ and $t$ commute and the statement is trivial.
If $t_1 \neq 0$ then $F[x,t]=K[x,t_1]$ where $K$ is the field generated over $F$ by $t_1^2$ and $x t_0$. Clearly $K[x,t_1]$ is a quaternion algebra over $K$.
If $K=F$ then $A=F[x,t] \otimes A_0$. Now $A_0$ is of degree $2^{n-1}$ and therefore it contains a field extension of degree $2^{n-1}$ of the center, and since $F$ is a 2-field it contains a quadratic field extension of $F$, which is generated by a square-central element. This element commutes with $x$ and $t$.
If $K \neq F$ then it must be of degree a power of 2.
Since $F$ is a 2-field, $K$ must contain a quadratic field extension of $F$ which is generated by a square-central element. This element commutes with $x$ and $t$.
\end{proof}

\section{Fields of cohomological dimension 2}\label{Secdec}

The following fact appeared in \cite[Theorem 3]{Kahn} under the assumption $\operatorname{char}(F)\neq 2$:

\begin{thm}\label{Merkurjev}
Let $F$ be a field with $\operatorname{cd}_2(F) \leq 2$. For any form $f \in I_q^2 F$, $E(f)$ is a division algebra if and only if $f$ is anisotropic.
As a result, every algebra of exponent $2$ and power $2^m$ is isomorphic to $E(f)$ for some $f \in I_q^2 F$, and therefore decomposes as the tensor product of quaternion algebras.
\end{thm}

\begin{proof}
By the remark preceding this theorem, we can assume $\operatorname{char}(F)=2$.
First we show that every from $f \in I_q^2 F$ is universal (i.e. represents every element of $F$):
Let $D(f)$ denote the set of nonzero elements represented by $f$. Since $\operatorname{cd}_2(F) \leq 2$ and $f \perp b f \in I_q^3 F$ for any $b \in F^\times$, we have $f \simeq b f$. This means $F^\times D(f) \subseteq D(F)$ and so $D(f)=F^\times$.

Now we show that if $f \in I_q^2 F$ is an anisotropic form of dimension at least $6$ and $f_K$ is isotropic for some $K=F[x: x^2+x=a]$ then $f_K=\varmathbb{H} \perp f_0$ where $\varmathbb{H}$ is a hyperbolic plane and $f_0$ is anisotropic:
Since $f_K$ is isotropic, $f=(f' \otimes [1,a]) \perp f_0$ for some bilinear form $f'$ and anisotropic form $f_0$ (see \cite[Proposition 34.8]{EKM}).
The dimension of $f'$ is equal to the Witt index of $f_K$.
If the dimension of $f'$ is greater than $1$ then it contains a subform $\phi$ of dimension $2$, and then $\phi \otimes [1,a]$ is a proper subform of $f$.
This proper subform is in $I_q^2 F$, which means that it is universal, and so $f$ is isotropic, contradiction.

Let $f$ be a form in $I_q^2 F$.
Clearly if $f$ is isotropic then $E(f)$ is not a division algebra.
Therefore assume that $f$ is anisotropic.
We want to show that $E(f)$ is a division algebra.
According to Springer's classical theorem (see \cite[Corollary 18.5]{EKM}) $f$ remains anisotropic under odd degree extensions.
Therefore $f_{F'}$ is anisotropic where $F'$ is the odd closure of $F$. Since $E(f_{F'})=E(f) \otimes F'$, it is enough to prove that $E(f_{F'})$ is a division algebra.
Assume $F$ is a 2-field then.

If $E(f)$ is split then since $I_q^2 F \cong Br_2(F)$, $f$ is hyperbolic.
Assume then that $E(f)$ is nonsplit.
Hence the dimension of $f$ is at least 4.
If $f$ is of dimension $4$ then it is known that $f$ is the norm form of the algebra $E(f)$, which is a quaternion algebra, and according to \cite[Corollary 12.5]{EKM} a quaternion algebra is a division algebra if and only if its norm form is anisotropic.
Therefore assume the dimension of $f$ is $2 m \geq 6$.

We use induction on $m$. 
Since $F$ is a 2-field with $\operatorname{cd}_2(F) \leq 2$, any finite field extension $K/F$ is also a 2-field with $\operatorname{cd}_2(K) \leq 2$.
The induction hypothesis is that for each $2 \leq t<m$ and every 2-field $L$ with $\operatorname{cd}_2(L) \leq 2$, if $\phi$ is an anisotropic form in $I_q^2 L$ of dimension $2 t$ then $E(\phi)$ is a division algebra.
Since $E(f)$ is nonsplit, it is a matrix algebra over a nontrivial division algebra of degree a power of $2$ over $F$. This algebra contains a nontrivial field extension of $F$, and since $F$ is a 2-field, this field extension contains a quadratic field extension $K$ of $F$. 
Since the dimension of $f$ is at least 6, the dimension of the anisotropic part of $f_K$ is $2(m-1)$.
Now, $K$ is also a 2-field with $\operatorname{cd}_2(K) \leq 2$, so by the induction hypothesis, $E(f_K)$ is of index $2^{m-2}$.
On the other hand, the index of $E(f)$ is at least twice the index of $E(f) \otimes K=E(f_K)$, which means that it must be $2^{m-1}$.
Consequently, $E(f)$ is a division algebra.
\end{proof}

The following theorem extends \cite[Theorem 3.3]{Barry} which presents a similar decomposition on the symbol-level for division algebras of exponent 2 over fields of cohomological 2-dimension 2 and characteristic not 2:

\begin{thm}\label{decompose}
Let $F$ be a field with $\operatorname{cd}_2(F) \leq 2$, and $A$ be a division algebra over $F$ of exponent 2 and degree $2^n$ for some $n \geq 2$. Let $x$ and $x'$ be two commuting elements in $A$ with $F[x] \neq F[x']$.
If $\operatorname{char}(F)=2$ then assume $x$ is either Artin-Schreier or square-central and $x'$ is Artin-Schreier.
If $\operatorname{char}(F)\neq 2$ assume that they are both square-central elements. Then $A=Q_1 \otimes Q_2 \otimes \dots \otimes Q_n$ such that $x \in Q_1$ and $x' \in Q_2$.
\end{thm}

\begin{proof}
Let $B$ be the centralizer of $F[x,x']$ in $A$.
It is enough to prove that $B=F[x,x'] \otimes B'$ for some central simple algebra $B'$ over $F$,
because then $A=A_0 \otimes B'$ (see \cite[Theorem 1.5]{BOI}) where $A_0=C_A(B)$ is a degree $4$ central simple algebra of exponent $\leq 2$ containing $x$ and $x'$ and therefore it decomposes as $Q_1 \otimes Q_2$ such that $x \in Q_1$ and $x' \in Q_2$ according to Lemma \ref{deg4}, and $B'$ decomposes as $Q_3 \otimes \dots \otimes Q_n$ according to Theorem \ref{Merkurjev}.

Let $f$ be the unique quadratic form in $I_q^2 F$ with $E(f) \cong A$.
In particular, $f$ is anisotropic.
However, $f_{F[x]}$ is isotropic, because $A \otimes F[x]$ is not a division algebra.
Since quadratic extensions are excellent, there exists $g \in I_q F$ such that $g_{F[x]}$ is isometric to the anisotropic part of $f_{F[x]}$, and according to Lemma \ref{exact} we can assume $g \in I_q^2 F$.

Now, $E(g_{F[x]})=C_A(F[x])$ and $g_{F[x]}$ is anisotropic.
However $g_{F[x,x']}$ is isotropic.
Assume $\operatorname{char}(F)=2$.
Then $g_{F[x]}=\phi \perp d [1,x'^2+x']$ for some $d \in F[x]$ and $\phi \in I_q F[x]$ such that $\phi_{F[x,x']}$ is the anisotropic part of $g_{F[x,x']}$.
Since $\operatorname{cd}_2(F) \leq 2$, $g_{F[x]} \perp d g_{F[x]}$ is hyperbolic, which means that $g_{F[x]} \simeq d g_{F[x]}$, and hence we can assume $d=1$.
Therefore $\phi$ is isometric to the anisotropic part of $g_{F[x]} \perp [1,x'^2+x']$.
Since $g \perp [1,x'^2+x']$ is in $I_q F$ and $F[x]/F$ is excellent, there exists $\phi'$ in $I_q F$ such that $\phi'_{F[x]}$ is isometric to $\phi$.
Now, $\phi'_{F[x,x']}$ is in $I_q^2 F[x,x']$, and therefore according to Lemma \ref{exact} there exists $\tau$ in $I_q^2 F$ such that $\tau_{F[x,x']}$ is isometric to $\phi'_{F[x,x']}$.
Consequently, $B$ is a restriction of $E(\tau)$ to $F[x,x']$.

Assume $\operatorname{char}(F) \neq 2$.
Then $g_{F[x]}=\phi \perp d \langle 1,-x'^2 \rangle$ for some $d \in F[x]$ and $\phi \in I_q F[x]$ such that $\phi_{F[x,x']}$ is the anisotropic part of $g_{F[x,x']}$.
Since $\operatorname{cd}_2(F) \leq 2$, $g_{F[x]} \perp -d g_{F[x]}$ is hyperbolic, which means that $g_{F[x]} \simeq d g_{F[x]}$, and hence we can assume $d=1$.
Therefore $\phi$ is isometric to the anisotropic part of $g_{F[x]} \perp \langle -1,x'^2 \rangle$.
Since $g \perp \langle -1,x'^2 \rangle$ is in $I_q F$ and $F[x]/F$ is excellent, there exists $\phi'$ in $I_q F$ such that $\phi'_{F[x]}$ is isometric to $\phi$.
Now, $\phi'_{F[x,x']}$ is in $I_q^2 F[x,x']$, and therefore according to Lemma \ref{exact} there exists $\tau$ in $I_q^2 F$ such that $\tau_{F[x,x']}$ is isometric to $\phi'_{F[x,x']}$.
Consequently, $B$ is a restriction of $E(\tau)$ to $F[x,x']$.
\end{proof}

\begin{cor}\label{charnot2anti}
If $\operatorname{char}(F)=2$ then for every two commuting Artin-Schreier elements $x,x' \in A$ with $F[x] \neq F[x']$ there exists a square-central element $z$ such that $x z+z x=x' z+z x'=z$.
If $\operatorname{char}(F)\neq 2$, for every two commuting square-central elements $x,x' \in A$ there exists a square-central element $z$ anti-commuting with them both.
\end{cor}

\begin{proof}
Since $A=Q_1 \otimes Q_2 \otimes \dots \otimes Q_n$ where $x \in Q_1$ and $x' \in Q_2$, there exist square-central elements $y \in Q_1$ and $y' \in Q_2$ such that $y^2,y'^2 \in F$ and $x y+y x=y$ and $x' y'+y' x'=y'$ if $\operatorname{char}(F)=2$, and $y x=-x y$ and $y' x'=-x' y'$ if $\operatorname{char}(F) \neq 2$. Take $z=y y'$.
\end{proof}

\begin{cor}\label{char2anti}
If $\operatorname{char}(F)=2$, $x \in A$ is square-central, $x' \in A$ is Artin-Schreier and $x x'=x' x$, then there exist a square-central element $z$ and an Artin-Schreier element $w$ such that $w x+x w=x$ and $w z+z w=z=x' z+z x'=z$.
\end{cor}

\begin{proof}
Since $A=Q_1 \otimes Q_2 \otimes \dots \otimes Q_n$ where $x \in Q_1$ and $x' \in Q_2$, there exists an Artin-Schreier element $w \in Q_1$ which satisfies $w x+x w=x$ and $w x'=x' w$. From here on we can apply Corollary \ref{charnot2anti}.
\end{proof}

From now on assume we focus on 2-fields.

\begin{thm}
Let $F$ be a 2-field with $\operatorname{cd}_2(F) \leq 2$, and $A$ be a division algebra over $F$ of exponent 2 and degree $2^n$ for some $n \geq 2$. If $\operatorname{char}(F)=2$ then for every two Artin-Schreier elemetns $x$ and $x'$ in $A$ there exists either a chain $x,y_1,x_1,y_2,x'$ such that $x_1$ is Artin-Schreier, $y_1$ and $y_2$ are square-central and $x y_1+y_1 x=x_1 y_1+y_1 x_1=y_1$ and $x' y_2+y_2 x'=x_1 y_2+y_2 x_1=y_2$, or a chain $x,y_1,x_1,y_2,x_3,y_3,x'$ with similar properties.
If $\operatorname{char}(F)\neq 2$ then for every two square-central elements $x$ and $x'$ in $A$ there exists a chain of square-central elements $x=x_0,x_1,x_2,x_3,x_4=x'$ such that $x_i x_{i+1}=-x_{i+1} x_i$ for each $0 \leq i \leq 3$.
\end{thm}

\begin{proof}
Assume $\operatorname{char}(F) \neq 2$.
According to Lemma \ref{charnot2commuting}, there exists a square-central element $x_2$ which commutes with them both.
According to Corollary \ref{charnot2anti}, there exists a square-central element $x_1$ commuting with $x$ and $x_2$, and a square-central element $x_3$ commuting with $x_2$ and $x'$.

The proof in characteristic 2 is similar, making use of Corollaries \ref{charnot2anti} and \ref{char2anti}.
\end{proof}

We say that two quaternion algebras $Q$ and $Q'$ over $F$ share a common slot if for some $a,b,c \in F$, either $Q=[a,b)$ and $Q'=[a,c)$ or $Q=[a,b)$ and $Q'=[c,b)$ if $\operatorname{char}(F)=2$, or $Q=(a,b)$ and $Q'=(a,c)$ if $\operatorname{char}(F) \neq 2$.
We say that two tensor products of quaternion algebras $\otimes_{i=1}^n Q_i$ and $\otimes_{i=1}^n Q'_i$ share a common slot if there exist $i$ and $j$ such that $Q_i$ and $Q'_j$ share a common slot.

\begin{thm}
\sloppy Let $F$ be a 2-field with $\operatorname{cd}_2(F) \leq 2$. For every two isomorphic tensor products of quaternion algebras over $F$, $\otimes_{i=1}^n Q_i$ and $\otimes_{i=1}^n Q_i'$, there exists a chain $\otimes_{i=1}^n Q_i,\otimes_{i=1}^n Q_i'',\otimes_{i=1}^n Q_i''',\otimes_{i=1}^n Q_i'$ such that every two adjacent tensor products share a common slot.
\end{thm}

\begin{proof}
If $\otimes_{i=1}^n Q_i$ is not a division algebra then it is isomorphic to $M_2(F) \otimes Q_2'' \otimes \dots \otimes Q_n''$ for some quaternion algebras $Q_i''$.
If $\operatorname{char}(F)=2$, write $Q_1=[a,b)$ and $Q_1'=[a',b')$ for some $a,a' \in F$ and $b,b' \in F^\times$.
Then $M_2(F)$ is isomorphic to both $[a,1)$ and $[a',1)$.
If $\operatorname{char}(F)\neq 2$, write $Q_1=(a,b)$ and $Q_1'=(a',b')$ for some $a,a',b,b' \in F^\times$.
Then $M_2(F)$ is isomorphic to both $(a,1)$ and $(a',1)$.
The required chain is obtained as a result.

Assume $\otimes_{i=1}^n Q_i$ is a division algebra. Write $Q_1=F[x,y : x^2+x=a, y^2=b, x y+y x=y]$ and $Q_2=F[x',y' : x'^2+x'=a', y'^2=b', x' y'+y' x'=y']$ if $\operatorname{char}(F)=2$, and $Q_1=F[x,y : x^2=a, y^2=b, x y=-y x]$ and $Q_1'=F[x',y' : x'^2=a, y'^2=b, x' y'=-y' x']$ if $\operatorname{char}(F)\neq 2$.
If $\operatorname{char}(F)=2$ then there exists an element $z$, either square-central or Artin-Schreier, commuting with $x$ and $x'$.
If $\operatorname{char}(F)\neq 2$ there exists a square-central element $z$ commuting with $x$ and $x'$.
Therefore, according to Theorem \ref{decompose}, $A$ is isomorphic to $\otimes_{i=1}^n Q_i''$ where $x \in Q_1''$ and $z \in Q_2''$, and $A$ is also isomorphic to $\otimes_{i=1}^n Q_i'''$ where $x' \in Q_1'''$ and $z \in Q_2'''$.
The required chain is obtained as a result.
\end{proof}

\section*{Acknowledgements}
We thank Jean-Pierre Tignol and the anonymous referee whose remarks improved the quality of the paper considerably.

\section*{Bibliography}
\bibliographystyle{amsalpha}
\bibliography{bibfile}
\end{document}